\documentclass[11pt]{article}
\usepackage{enumerate}
\usepackage{amssymb,a4wide,latexsym,makeidx,epsfig,fleqn}
\usepackage{float}
\usepackage{amsthm}
\usepackage{amsmath}
\usepackage{enumerate}
\usepackage{color}
\usepackage{graphicx}
\usepackage{multirow}
\usepackage{makecell}
\usepackage{caption}
\usepackage{subcaption}
\DeclareCaptionLabelFormat{continued}{Ðø #1 #2}
\newtheorem{theorem}{Theorem}[section]
\newtheorem{lemma}[theorem]{Lemma}

\newtheorem{problem}[theorem]{Problem}
\newtheorem{claim}{Claim}

\newcommand{\tabincell}[2]{\begin{tabular}{@{}#1@{}}#2\end{tabular}}
\allowdisplaybreaks

\begin{document}
\textwidth 150mm \textheight 225mm

\title{Extremal Laplacian energy of $\overrightarrow{C_{k+1}}$-free digraphs\thanks{\small Supported by the Natural Science Foundation of Shaanxi Province (No. 2025JC-YBQN-108), the China Postdoctoral Science Foundation (No. 2025M783120), and Shandong Postdoctoral Science Foundation (No. SDZZ-ZR-202501340).}}

\author{{Xiuwen Yang$^{a}$, Lin-Peng Zhang$^{b,}$\footnote{Corresponding author.}}\\
{\small $^{a}$ {School of Science, Xi'an University of Posts \& Telecommunications,}}\\{\small Xi'an, Shaanxi 710121, P.R. China.}\\ {\small $^{b}$ School of Mathematics, Shandong University,}\\{\small Jinan, Shandong 250100, P.R. China.}\\{\small E-mail: yangxiuwen1995@163.com, lpzhangmath@163.com }}
\date{}
\maketitle
\begin{center}
\begin{minipage}{135mm}
\vskip 0.3cm
\begin{center}
{\small {\bf Abstract}}
\end{center}
{\small
The Laplacian energy of a digraph $G$ is defined as $\sum_{i=1}^n \lambda_i^2$, where $\lambda_i$ are the eigenvalues of the Laplacian matrix of $G$. A (di)graph $G$ is said to be $H$-free if it does not contain a copy of the fixed (di)graph $H$ as a sub(di)graph. In this paper, we extend the Tur\'{a}n problems to spectral Tur\'{a}n problems in digraphs: what is the maximal Laplacian energy of an $H$-free digraph of given order? In particular, we determine the maximum Laplacian energy and characterize the extremal digraphs of $\overrightarrow{C_{k+1}}$-free digraphs.
}
\end{minipage}
\end{center}

\section{Introduction}

Before delving into essential terminology and notation, we first offer a brief introduction to establish the necessary background and motivation for our paper.

\subsection{Background and motivation}

Our work is motivated by recent advances in extremal graph theory, particularly the rapidly growing field of spectral Tur\'{a}n problem. A (di)graph $G$ is said to be $H$-free if it does not contain a copy of the fixed (di)graph $H$ as a sub(di)graph. The Tur\'{a}n problem is a classical extremal problem in extremal graph theory that considers the maximum size of an $H$-free graph of given order.
In 2010, Nikiforov~\cite{Ni10} proposed a spectral Tur\'{a}n problem: what is the maximal spectral radius of a $H$-free graph of given order? This is also known as Brualdi-Solheid-Tur\'{a}n type problem. In early work, Wilf~\cite{Wi} investigated the problem of maximum spectral radius of $K_{r+1}$-free graphs of given order. Nikiforov~\cite{Ni08} studied the maximum spectral radius problem for graphs forbidding odd cycles, while Nikiforov~\cite{Ni10} further explored the case of forbidden paths. Furthermore, Nikiforov~\cite{Ni10} proposed conjectures regarding the maximum spectral radius for graphs forbidding%%excluding 
cycles and trees. Recently, Cioab\u{a} et al.~\cite{CiFT} determined the maximum spectral radius and the corresponding extremal graphs for graphs forbidding friendship graph. Zhai et al.~\cite{ZhLS} resolved the same problem for $K_{2,r+1}$-free graphs of given size. Later, Zhai and Lin~\cite{ZhLi22} characterized the extremal graphs attaining the maximum spectral radius among graphs forbidding $K_{s,t}$-minors. Cioab\u{a} et al.~\cite{CiDT} investigated the spectral extremal problem for graphs forbidding even cycles. More results on spectral Tur\'{a}n problems see~\cite{GoHo,LiNi1,LiNi2,LiSY,LiNW,Ni02,Ni08,No,ZhLi20,ZhLi23,ZhSh}. But for digraphs, there are only a few results about spectral Tur\'{a}n problem. In~\cite{BrLi,Dr}, Brualdi, Li and Drury %the authors 
determined the maximal spectral radius and extremal digraph of $\overrightarrow{C_2}$-free digraphs: if $n$ is odd, the extremal digraph is a tournament which has indegree and outdegree $\frac{n-1}{2}$ at each vertex; if $n$ is even, the extremal digraph is a tournament which is isomorphic to the Brualdi-Li tournament.

All the aforementioned results concern the spectral radius of adjacency matrices of graph. %They consider the maximum adjacency spectral radius problem of given order or size.%
This leads to the natural extremal question: what is the maximal adjacency spectral radius of an $H$-free digraph of given order or size? However, the present work primarily investigates the Laplacian energy of digraphs. As is known, the classical Tur\'{a}n problem and its spectral analogue are intrinsically related for graphs. A natural question arises: does a similar relation hold for digraphs? While this is a compelling research direction, it remains largely unexplored. However, for the Laplacian energy of digraphs, which is fundamentally tied to the sum of squares of outdegrees (and thus to the arc count), a link to the digraph Tur\'{a}n problem appears more plausible. This possibility forms the primary motivation for our work: to investigate the spectral Tur\'{a}n problem for Laplacian energy by asking what is the maximum Laplacian energy of an $H$-free digraph with given order or size?

%As we known, on graphs, the Tur\'{a}n problem and the spectral Tur\'{a}n problem are related in some way. So are they related on digraphs? Although this is clearly an important research direction, the answer remains unknown. But for Laplacian energy of digraphs, we know that the Laplacian energy of digraphs is related to the sum of the squares of the outdegrees, and the sum of the outdegrees is equal to the number of arcs. So, is the maximum Laplacian energy problem related to the digraph Tur\'{a}n problem? This constitutes the primary motivation for our investigation of the spectral Tur\'{a}n problems concerning Laplacian energy of digraphs: what is the maximal Laplacian energy of an $H$-free digraph of given order?

Classic Tur\'{a}n problems form a cornerstone of extremal combinatorics, originating from Tur\'{a}n's generalization of Mantel's theorem. We recall the classic Tur\'{a}n theorem: for any integers $n,k\ge 1$, the unique $K_{k+1}$-free graph on 
$n$ vertices with the maximum number of edges is the Tur\'{a}n graph $T_{n,k}$—the complete $k$-partite graph with parts as equal as possible. More generally, for a given (di)graph $H$ and $n\in \mathbb{N}^+$, we denote by $\textnormal{ex}(n,H)$ the Tur\'{a}n number (the maximum number of edges/arcs in an $H$-free (di)graph of order $n$), and by 
$\textnormal{EX}(n,H)$ the family of all extremal (di)graphs attaining this bound.

%In graph theory, Tur\'{a}n problems represent a fundamental class of extremal problems, tracing their origins to Tur\'{a}n's generalization of Mantel's theorem. We only introduce the Tur\'{a}n theorem here. For complete graphs $K_{k+1}$, Tur\'{a}n's theorem \cite{Ju} obtained the unique $K_{k+1}$-free extremal graph of order $n$ is the Tur\'{a}n graph $T_{n,k}$, where the Tur\'{a}n graph is a complete $k$-partite graph with an almost balanced partition on $n$ vertices. For any given graphs (or digraphs) $H$ and $n\in \mathbb{N}^+$, let $\textnormal{ex}(n,H)$ denote the Tur\'{a}n number and $\textnormal{EX}(n,H)$ denote the family of all $H$-free extremal graphs (or digraphs).

\noindent\begin{theorem}\label{th:exK}(\cite{Ju}) Let $k,n\in \mathbb{N}^+$, $n=qk+r$, $0\leq r< k$. Then
$$\textnormal{ex}(n,K_{k+1})=\frac{k-1}{2k}n^2-\frac{r(k-r)}{2k},$$
and $\textnormal{EX}(n,K_{k+1})=\{T_{n,k}\}$.
\end{theorem}

Although the Tur\'{a}n problem has been extensively studied for graphs, its extension to digraphs remains relatively unexplored.
%While most classical Tur\'{a}n results have been established for graphs, corresponding studies for digraphs remain relatively scarce. 
Brown and Harary~\cite{BrHa} resolved Tur\'{a}n problems for complete digraphs $\overleftrightarrow{K_{k+1}}$ and tournaments $\overrightarrow{T_{k+1}}$.
\noindent\begin{theorem}\label{th:exCT}(\cite{BrHa}) Let $k,n\in \mathbb{N}^+$. Then
$$\textnormal{ex}(n,\overleftrightarrow{K_{k+1}})=\binom{n}{2}+\textnormal{ex}(n,K_{k+1})\ \
\textnormal{and}\ \ \textnormal{ex}(n,\overrightarrow{T_{k+1}})=2\textnormal{ex}(n,K_{k+1}).$$
\end{theorem}

Howalla et al.~\cite{HoDT1,HoDT2} determined the maximum number of arcs in digraphs forbidding $k$ directed paths with common endpoints for $k=2,3$. Huang and Lyu~\cite{HuLy} investigated Tur\'{a}n problems for oriented $C_4$ with specific edge directions. Zhou and Li~\cite{ZhLi} characterized the maximum size and extremal %maximum 
digraphs for forbidden directed paths, directed cycles, and another variant of oriented $C_4$. 

As we discussed before, this work
is devoted to investigating spectral Tur\'{a}n problems concerning Laplacian energy of $\overrightarrow{C_{k+1}}$-free digraphs.
In 1970,  Gutman~\cite{Gut1} first introduced the energy of adjacency matrix of a graph as the sum of the absolute values of the eigenvalues. More recently, in 2006, Gutman and Zhou~\cite{GuZh} defined the Laplacian energy of a graph as the sum of the absolute values of the differences between Laplacian eigenvalues and the average degree. In the same year, Lazi\'c~\cite{La} defined the Laplacian energy of a graph as the sum of the squares of the Laplacian eigenvalues. In 2010, Perera and Mizoguchi~\cite{PeMi} extended the Lazi\'c version's definition to digraphs. In 2015, Qi et al.~\cite{QiFL} obtained lower and upper bounds on the Laplacian energy of digraphs and also characterized the extremal digraphs.
In 2020, Yang and Wang~\cite{YaWa} determined the directed trees, unicyclic digraphs and bicyclic digraphs which attain the maximal and minimal Laplacian energy among all digraphs with $n$ vertices, respectively. In 2024, Yang et al.~\cite{YaBW} characterized the digraphs which attain the minimal and maximal Laplacian energy within classes of digraphs with a fixed dichromatic number. We refer the interested reader to the three monographs~\cite{Gut2,GuLi,LiSG} for a wide range of alternative definitions for energies of graphs and digraphs, and a wealth of references to obtained results.

Before we present our results and proofs, we will recall some of the essential terminology and notation, and give some additional background and related results.

\subsection{Terminology and notation}

For a digraph $G$, we use $\mathcal{V}(G)$ and $\mathcal{A}(G)$ to denote the vertex set and arc set of $G$, respectively, and we use $n=|\mathcal{V}(G)|$ and $e=|\mathcal{A}(G)|$ to denote the order and size of $G$, respectively. We denote an arc from a vertex $u$ to a vertex $v$ by $(u,v)$, and we call $u$ the tail and $v$ the head of the arc $(u,v)$. For a vertex $v\in \mathcal{V}(G)$, the outdegree $d_G^+(v)$ is the number of arcs in $\mathcal{A}(G)$ whose tail is $v$, while the indegree $d_G^-(v)$ is the number of arcs in $\mathcal{A}(G)$ whose head is $v$. For two disjoint subdigraphs $G_1, G_2\subseteq G$, we write $G_1\rightarrow G_2$ if $(u,v)\in\mathcal{A}(G)$ for every $u\in\mathcal{V}(G_1)$ and $v\in\mathcal{V}(G_2)$, and $G_1\nrightarrow G_2$ if $(u,v)\notin\mathcal{A}(G)$ for every $u\in\mathcal{V}(G_1)$ and $v\in\mathcal{V}(G_2)$. We also use $G_1\mapsto G_2$ to denote $G_1\rightarrow G_2$ and $G_1\nrightarrow G_2$. correct? maybe $G_2\nrightarrow G_1$?

A directed walk $\pi$ of length $\ell$ from vertex $u$ to vertex $v$ in $G$ is a sequence of vertices $\pi$: $u=v_0,v_1,\ldots,v_{\ell}=v$, where $(v_{k-1},v_k)$ is an arc of $G$ for any $1\leq k\leq \ell$. If $u=v$, then $\pi$ is called a directed closed walk. If all vertices of the directed walk $\pi$ of length $\ell$ are distinct, then we call it a directed path, and denote it by $\overrightarrow{P_{\ell+1}}$; a directed closed walk of length $\ell$ in which all except the end vertices are distinct is called a directed cycle, and denoted by $\overrightarrow{C_{\ell}}$. We let $c_2$ denote the total number of directed closed walks of length $2$. Throughout the remainder of the paper, we consider only connected digraphs without loops or multiple arcs.

The adjacency matrix $A(G)=(a_{ij})$ of $G$ is an $n\times n$ matrix whose $(i,j)$-entry equals $1$ if $(v_i, v_j)\in\mathcal{A}(G)$ and equals $0$ otherwise. The diagonal outdegree matrix $D^+(G)$ of $G$ is defined by $D^+(G)=diag(d_1^+,d_2^+,\ldots,d_n^+)$. The Laplacian matrix $L(G)$ of $G$ is defined by $L(G)=D^+(G)-A(G)$. Hence,
$L(G)=(\ell_{ij})$ is an $n \times n$ matrix, where
$$\ell_{ij}=
\begin{cases}
d_i^+,& \mbox{if} \ i=j,\\
-1,& \mbox{if} \ (v_i, v_j)\in\mathcal{A}(G),\\
0,& \mbox{otherwise}.
\end{cases}$$
Let $\lambda_1, \lambda_2, \ldots, \lambda_n$ are the eigenvalues of $L(G)$. Perera and Mizoguchi~\cite{PeMi} defined the Laplacian energy of a digraph as the sum of the squares of the Laplacian eigenvalues:
$$LE(G)=\sum_{i=1}^n \lambda_i^2.$$
Qi et al.~\cite{QiFL} gave the formula of Laplacian energy of digraphs:
$$LE(G)=\sum_{i=1}^n(d^+_i)^2+c_2.$$

Next, we introduce several common classes of digraphs.

A tournament is a digraph obtained from an undirected complete graph by assigning a direction to each edge. A transitive tournament is a tournament $G$ satisfying the following condition: if $(u,v)\in \mathcal{A}(G)$ and $(v,w)\in \mathcal{A}(G)$, then $(u,w)\in \mathcal{A}(G)$.

Let $\overset{\longleftrightarrow}{K_{n}}$ denote the complete digraph on $n$ vertices in which for two arbitrary distinct vertices $v_i,v_j\in  \mathcal{V}(\overset{\longleftrightarrow}{K_{n}})$, there are arcs $(v_i,v_j)$ and $(v_j,v_i)\in \mathcal{A}(\overset{\longleftrightarrow}{K_{n}})$.

A balanced complete bipartite digraph is a bipartite directed graph satisfying:
\begin{itemize}
  \item [(i)] the vertex set $\mathcal{V}$ is partitioned into two disjoint subsets $X$ and $Y$, with $\big||X|-|Y|\big|\leq1$;
  \item [(ii)] for every $x\in X$ and $y\in Y$, there are two arcs $(x,y)$ and $(y,x)$;
  \item [(iii)] there are no arcs between vertices within the same part $X$ or $Y$.
\end{itemize}

\subsection{Related work}

For any given graphs (or digraphs) $H$ and $n\in \mathbb{N}^+$, let $\textnormal{ex}_{LE}(n,H)$ denote the maximum Laplacian energy of $H$-free digraphs and $\textnormal{EX}_{LE}(n,H)$ denote the family of all digraphs which have the maximum Laplacian energy of $H$-free digraphs. As mentioned earlier, this paper aims to obtain the maximal Laplacian energy of an $\overrightarrow{C_{k+1}}$-free digraph of given order. Before presenting our results, we first introduce the Tur\'{a}n number of an $\overrightarrow{C_{k+1}}$-free digraph.

\begin{figure}[htbp]
\centering
\includegraphics[scale=1]{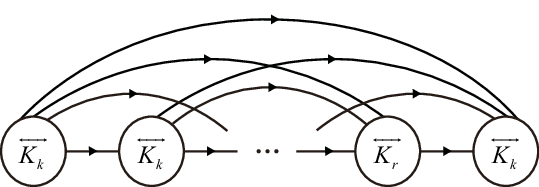}
\caption{Digraph $\protect\overrightarrow{F_{n,k}}$ ($q$ copies of $\protect\overleftrightarrow{K_k}$ and one copy of $\protect\overleftrightarrow{K_r}$)}\label{fi:Fnk}
\end{figure}

For integers $n,k\in \mathbb{N}^+$, $n=qk+r$, $0\leq r<k$, we define a class of digraphs on $n$ vertices, denoted by $\overrightarrow{F_{n,k}}$, which satisfies the following properties (set $q'=q$ if $r=0$ and $q'=q+1$ otherwise):
\begin{itemize}
  \item [(i)] $\mathcal{V}(\overrightarrow{F_{n,k}})$ has a partition $\{V_1, V_2, \ldots, V_{q'}\}$ such that $|V_i|=k$ for all but at most one $V_i$;
  \item [(ii)] $\overrightarrow{F_{n,k}}[V_i]$ is a complete digraph for $1\le i\le q'$;
  \item [(iii)] for any two parts $V_i$ and $V_j$ with $1\leq i<j\leq q'$, $V_i \mapsto V_j$ (see Figure~\ref{fi:Fnk}).
\end{itemize}

Note that there may be more than one nonisomorphic $\overrightarrow{F_{n,k}}$ for given $n, k$. However all $\overrightarrow{F_{n,k}}$ have the same size. We denote by $\overrightarrow{\mathcal{F}_{n,k}}$ the set of all $\overrightarrow{F_{n,k}}$. If $0< r<k$, let $\overrightarrow{{F}_{n,k}^{s+1}}\in\overrightarrow{\mathcal{F}_{n,k}}$ such that $\overrightarrow{F_{n,k}}[V_{s+1}]=\overleftrightarrow{K_r}$ and $\overrightarrow{F_{n,k}}[V_{i}]=\overleftrightarrow{K_k}$ for $i=1,\ldots,s,s+2,\ldots,q+1$, where $s=0,1,\ldots,q$. If $r=0$, then every $\overrightarrow{F_{n,k}}[V_i]=\overleftrightarrow{K_k}$ for $1\le i\le q$. By this time, $\overrightarrow{\mathcal{F}_{n,k}}$ has only one type and denoted by $\overrightarrow{F_{n,k}^0}$.

In 2023, Zhou and Li~\cite{ZhLi} determined the Tur\'{a}n number and the extremal digraphs for $\overrightarrow{C_{k+1}}$, $k\geq3$.

\noindent\begin{theorem}\label{th:exC}(\cite{ZhLi}) Let $k,n\in \mathbb{N}^+$, $n=qk+r$, $0\leq r<k$. Then
$$\textnormal{ex}(n,\overrightarrow{C_{k+1}})=\frac{1}{2}n^2+\frac{k-2}{2}n-\frac{r(k-r)}{2}.$$
Furthermore, $\textnormal{EX}(n,\overrightarrow{C_{k+1}})=\overrightarrow{\mathcal{F}_{n,k}}$ for $k\geq3$.
\end{theorem}

We know that the Laplacian energy of digraphs is related to the sum of the squares of the outdegrees, and the sum of the outdegrees is equal to the number of arcs. So it's natural for us to conjecture that the digraph which has the maximum Laplacian energy of $\overrightarrow{C_{k+1}}$-free digraphs is also in $\overrightarrow{\mathcal{F}_{n,k}}$ for $k\geq3$.
In particular, we obtain the following result.

\noindent\begin{theorem}\label{th:exLEC} Let $k,n\in \mathbb{N}^+$, $n=qk+r$, $0\leq r<k$ and $k\geq 3$.
$$\textnormal{ex}_{LE}(n,\overrightarrow{C_{k+1}})
=\frac{1}{3}n^3+\left(\frac{k}{2}-1\right)n^2+\frac{k^2}{6}n+\frac{2}{3}r^3-\frac{k}{2}r^2-\frac{k^2}{6}r,$$
and
$$\textnormal{EX}_{LE}(n,\overrightarrow{C_{k+1}})=
\begin{cases}
\overrightarrow{{F}_{n,k}^{q+1}},& \textnormal{if} \ 0<r<k,\\
\overrightarrow{{F}_{n,k}^0},& \textnormal{if} \ r=0.
\end{cases}$$
\end{theorem}

If $k=1$ or $k=2$, then $\overrightarrow{C_2}=\overleftrightarrow{K_2}$ and $\overrightarrow{C_3}$ is a tournament. In 1970, Brown and Harary~\cite{BrHa} already obtained the Tur\'{a}n number for complete digraphs and tournaments, see Theorem~\ref{th:exCT}. $\overrightarrow{\mathcal{F}_{n,1}}$ and $\overrightarrow{\mathcal{F}_{n,2}}$ are also among the extremal digraphs of Tur\'{a}n number for $\overrightarrow{C_2}$-free digraphs and $\overrightarrow{C_3}$-free digraphs, respectively. Then, when $k=1$ or $k=2$, how can the $\overrightarrow{C_{k+1}}$-free digraph that has the maximum Laplacian energy be characterized?

If $k=1$, $\textnormal{EX}(n,\overrightarrow{C_2})$ is a tournament obviously. We have the following result.

\noindent\begin{theorem}\label{th:exLEC2} Let $n\in \mathbb{N}^+$. Then
$$\textnormal{ex}_{LE}(n,\overrightarrow{C_2})=\frac{n(n-1)(2n-1)}{6},$$
and $\textnormal{EX}_{LE}(n,\overrightarrow{C_2})$ is a transitive tournament.
\end{theorem}

If $k=2$, Brown and Harary~\cite{BrHa} obtained $\textnormal{EX}(n,\overrightarrow{C_3})=\overrightarrow{\mathcal{BK}_{n,p}}$. The definition of $\overrightarrow{\mathcal{BK}_{n,p}}$ is as follows:

For integers $n, n_1,n_2,\ldots,n_p\in \mathbb{N}^+$, $\sum^p_{i=1} n_i=n$, we define a class of digraphs on $n$ vertices, denoted by $\overrightarrow{BK_{n,p}}$, which satisfies the following properties:
\begin{itemize}
  \item [(i)] $\mathcal{V}(\overrightarrow{BK_{n,p}})$ has a partition $\{V_1, V_2, \ldots, V_{p}\}$ such that $|V_i|=n_i$ is an even for all but at most one $V_i$;
  \item [(ii)] $\overrightarrow{BK_{n,p}}[V_i]$ is a balanced complete bipartite digraph for $1\le i\le p$;
  \item [(iii)] for any two parts $V_i$ and $V_j$ with $1\leq i<j\leq p$, $V_i \mapsto V_j$.
\end{itemize}

Note that there may be more than one nonisomorphic $\overrightarrow{BK_{n,p}}$ for given $n$ and unspecified $p$. However all $\overrightarrow{BK_{n,p}}$ have the same size. We denote by $\overrightarrow{\mathcal{BK}_{n,p}}$ the set of all $\overrightarrow{BK_{n,p}}$. Obviously, $\overrightarrow{\mathcal{F}_{n,2}}\subseteq \overrightarrow{\mathcal{BK}_{n,p}}$.

In particular, if $|V_i|=4$ or $2$ for all $i\in[1,p]$, we denote by $\overrightarrow{BK^0_{n,p}}$; if $|V_i|=4$ or $2$ for $i\in[1,p-1]$ and $|V_p|=3$ or $1$, we denote by $\overrightarrow{BK^1_{n,p}}$. Also, let $\overrightarrow{\mathcal{BK}^0_{n,p}}$ ($\overrightarrow{\mathcal{BK}^1_{n,p}}$)be the set of all $\overrightarrow{BK^0_{n,p}}$ ($\overrightarrow{BK^1_{n,p}}$). And $\overrightarrow{{F}_{n,2}^0}\in\overrightarrow{\mathcal{BK}^0_{n,p}}$, $\overrightarrow{{F}_{n,2}^{\frac{n+1}{2}}}\in \overrightarrow{\mathcal{BK}^1_{n,p}}$.

Hence, if $k=2$, we have the following result.

\noindent\begin{theorem}\label{th:exLEC3} Let $n\in \mathbb{N}^+$, $n=2q+r$, $0\leq r<2$. Then
$$\textnormal{ex}_{LE}(n,\overrightarrow{C_3})=\frac{2q}{3}(3n^2-6qn+4q^2+2),$$
and
$$\textnormal{EX}_{LE}(n,\overrightarrow{C_3})=
\begin{cases}
\overrightarrow{\mathcal{BK}^1_{n,p}} ,& \textnormal{if} \ r=1,\\
\overrightarrow{\mathcal{BK}^0_{n,p}},& \textnormal{if} \ r=0.
\end{cases}$$
\end{theorem}

\medskip\medskip
We also used an extremely important tool: Karamata's inequality.

\medskip\noindent\textbf{Karamata's inequality:}

Let $I$ be an interval of the real line and let $f$ denote a real-valued, convex function defined on $I$. If $x_1,x_2,\ldots,x_n$ and $y_1,y_2,\ldots,y_n$ are numbers in $I$ such that $(x_1,x_2,\ldots,x_n)$ majorizes $(y_1,y_2,\ldots,y_n)$, then
$$f(x_1)+f(x_2)+\cdots+f(x_n)\geq f(y_1)+f(y_2)+\cdots+f(y_n).$$
If $f$ is a strictly convex function, then the inequality holds with equality if and only if we have $x_i=y_i$ for all $i=1,2,\ldots,n$.

Here majorization means that $x_1,x_2,\ldots,x_n$ and $y_1,y_2,\ldots,y_n$ satisfies
$$x_1\geq x_2\geq \cdots \geq x_n\ \textnormal{and}\ y_1\geq y_2\geq \cdots \geq y_n,$$
and we have the inequalities
$$x_1+x_2+\cdots+x_i\geq y_1+y_2+\cdots+y_i,$$
for all $i=1,2,\ldots,n-1$, and the equality
$$x_1+x_2+\cdots+x_n=y_1+y_2+\cdots+y_n.$$

\medskip

The rest of the paper is organized as follows. In Section~\ref{sec2}, we present the proofs of Theorems~\ref{th:exLEC2} and~\ref{th:exLEC3}. In Section~\ref{sec3}, we give the proof of Theorem~\ref{th:exLEC}. We finish the paper with some problems in Section~\ref{sec4}.

{\section{Extremal Laplacian energy of $\protect\overrightarrow{C_{k+1}}$-free digraphs when $k=1,2$}\label{sec2}}

In this section, we give the proofs of Theorems~\ref{th:exLEC2} and~\ref{th:exLEC3}. First, we consider the extremal Laplacian energy of $\overrightarrow{C_2}$-free digraphs.

\medskip
\noindent\textbf{Proof of Theorem~\ref{th:exLEC2}:}
Let $G$ be a $\overrightarrow{C_2}$-free digraph having the maximum Laplacian energy. From Theorems~\ref{th:exK} and \ref{th:exCT}, we has known
$$\textnormal{ex}(n,\overrightarrow{C_2})=\binom{n}{2}+\textnormal{ex}(n,K_2)=\binom{n}{2}+e(T_{n,1})=\frac{n(n-1)}{2}.$$
Actually, if $G$ is a $\overrightarrow{C_2}$-free digraph, tournament has a maximum number of arcs.

If $e(G)<\frac{n(n-1)}{2}$, there must exist the vertices $u,v$, such that $(u,v)\notin \mathcal{A}(G)$ and $(v,u)\notin \mathcal{A}(G)$. Let $G'=G+(u,v)$ or $G'=G+(v,u)$. Then $G'$ is also a $\overrightarrow{C_2}$-free digraph and $LE(G')\geq LE(G)$. So $G$ is a tournament.

We know that the outdegree descending sequence of transitive tournament is $\{n-1,n-2,\ldots,1,0\}$. If we change the directions of some arcs in the transitive tournament, the resulting graph will always be a tournament. Suppose the outdegree sequence of vertices of any tournament satisfies that $d_1^+\geq d_2^+\geq\cdots\geq d_n^+$. Then we can get
$$\sum_{i=1}^j(n-i)\geq \sum_{i=1}^j d_i^+.$$
So by Karamata's inequality,
$$\sum_{i=1}^j(n-i)^2\geq \sum_{i=1}^j (d_i^+)^2.$$
That is, transitive tournament has the maximum Laplacian energy in tournament.

Therefore, $\textnormal{ex}_{LE}(n,\overrightarrow{C_2})=\frac{n(n-1)(2n-1)}{6}$, and $\textnormal{EX}_{LE}(n,\overrightarrow{C_2})$ is a transitive tournament. This completes the proof of Theorem~\ref{th:exLEC2}.\qed

%$\hfill\square$

\medskip

For the Laplacian energy $LE(G)=\sum_{i=1}^n(d^+_i)^2+c_2$, let $M_1(G)=\sum_{i=1}^n(d^+_i)^2$ be the First Zagreb Index of $G$.
For any given graphs (or digraphs) $H$ and $n\in \mathbb{N}^+$, let $\textnormal{ex}_{M_1}(n,H)$ denote the maximum First Zagreb Index of $H$-free digraphs and $\textnormal{EX}_{M_1}(n,H)$ denote the family of all digraphs which have the maximum First Zagreb Index of $H$-free digraphs.
Before considering the extremal Laplacian energy of $\overrightarrow{C_3}$-free digraphs, we consider the extremal First Zagreb Index of $\overrightarrow{C_3}$-free digraphs firstly.

\noindent\begin{lemma}\label{le:exMC3} Let $n\in \mathbb{N}^+$, $n=2q+r$, $0\leq r<2$. Then
$$\textnormal{ex}_{M_1}(n,\overrightarrow{C_3})=\frac{2q}{3}(3n^2-6qn+4q^2-1),$$
and
$$\textnormal{EX}_{M_1}(n,\overrightarrow{C_3})=
\begin{cases}
\overrightarrow{{F}_{n,2}^{\frac{n+1}{2}}},& \textnormal{if} \ r=1,\\
\overrightarrow{{F}_{n,2}^0},& \textnormal{if} \ r=0.
\end{cases}$$
\end{lemma}
\begin{proof}
For any digraph $G$, let $M_1(G)=\sum_{i=1}^n(d^+_i)^2$ be the First Zagreb Index of $G$ and $Sd_t(G)$ be the sum of the $t$ largest outdegree of $G$. Assume that $G^\ast$ is a $\overrightarrow{C_3}$-free digraph having the maximum First Zagreb Index. We consider two cases: (i) $e(G^\ast)=\textnormal{ex}(n,\overrightarrow{C_3})$; (ii) $e(G^\ast)<\textnormal{ex}(n,\overrightarrow{C_3})$.

\medskip\medskip
\noindent(i) $e(G^\ast)=\textnormal{ex}(n,\overrightarrow{C_3})$.

From Theorems~\ref{th:exK} and \ref{th:exCT}, we have
$$\textnormal{ex}(n,\overrightarrow{C_3})=\textnormal{ex}(n,\overrightarrow{T_3})=2\textnormal{ex}(n,K_3)=\frac{1}{2}n^2-\frac{r(2-r)}{2},$$
where $r=0$ or $1$. That is,
$$\textnormal{ex}(n,\overrightarrow{C_3})=
\begin{cases}
\frac{n^2}{2}-\frac{1}{2},& \textnormal{if} \ r=1,\\
\frac{n^2}{2},& \textnormal{if} \ r=0.
\end{cases}$$
And
$$\textnormal{EX}(n,\overrightarrow{C_3})=\overrightarrow{\mathcal{BK}_{n,p}}.$$
But $\{\overrightarrow{{F}_{n,2}^{\frac{n+1}{2}}}, \ \overrightarrow{{F}_{n,2}^0}\}\subseteq\overrightarrow{\mathcal{F}_{n,2}}\subseteq\overrightarrow{\mathcal{BK}_{n,p}}$ is just one of the extremal digraph of Tur\'{a}n number for $\overrightarrow{C_3}$-free digraphs.

We consider $r=1$, thus, $q=\frac{n-1}{2}$. We start with the following claim.

\noindent\begin{claim}\label{cl:ch-1}
$\overrightarrow{{F}_{n,2}^{\frac{n+1}{2}}}$ is a $\overrightarrow{C_3}$-free digraph having the maximum First Zagreb Index.
\end{claim}

Let the vertex set of $\overrightarrow{{F}_{n,2}^{\frac{n+1}{2}}}$ is $\mathcal{V}(\overrightarrow{{F}_{n,2}^{\frac{n+1}{2}}})=\{v_1^1,v_1^2,v_2^1,v_2^2,\ldots,v_{\frac{n-1}{2}}^1,v_{\frac{n-1}{2}}^2,v_{\frac{n+1}{2}}^1\}$. For convenience, we denote $\overrightarrow{{F}_{n,2}^{\frac{n+1}{2}}}=F^{\frac{n+1}{2}}$. Then the outdegree descending sequence of $F^{\frac{n+1}{2}}$ is
$$\{n-1,n-1,n-3,n-3,n-5,n-5,\ldots,4,4,2,2,0\},$$
and
$$\begin{cases}
d_{F^{\frac{n+1}{2}}}^+(v_i^1)=d_{F^{\frac{n+1}{2}}}^+(v_i^2),\\
d_{F^{\frac{n+1}{2}}}^+(v_i^1)=d_{F^{\frac{n+1}{2}}}^+(v_{i+1}^1)+2,\\
d_{F^{\frac{n+1}{2}}}^+(v_1^1)=d_{F^{\frac{n+1}{2}}}^+(v_1^2)=n-1,\\
d_{F^{\frac{n+1}{2}}}^+(v_{\frac{n+1}{2}}^1)=0,
\end{cases}$$
where $i=1,2,\ldots,\frac{n-1}{2}$.

By Karamata's inequality, for any $\overrightarrow{C_3}$-free digraph $G$, if
$$Sd_t(F^{\frac{n+1}{2}}) \geq Sd_t(G)$$
for all $t=1,2,\ldots,n-1$ and $Sd_n(F^{\frac{n+1}{2}})=Sd_n(G)$, then
$$M_1(F^{\frac{n+1}{2}}) \geq M_1(G),$$
the inequality holds with equality if and only if $G=F^{\frac{n+1}{2}}$.
That is to say, if there exists a $\overrightarrow{C_3}$-free digraph $G$, such that the outdegree descending sequence of $F^{\frac{n+1}{2}}$ and  $G$ do not satisfy the majorization, then we might have $M_1(F^{\frac{n+1}{2}})<M_1(G)$. Next, we will prove the non-existence of such a $\overrightarrow{C_3}$-free digraph $G$ by performing arc transformations on $F^{\frac{n+1}{2}}$.

Perform the following operations on the arcs in $F^{\frac{n+1}{2}}$:
\begin{center}
\textbf{delete an arc with a tail of $v_{i_1}^{j_1}$ and add an arc with a tail of $v_{i_2}^{j_2}$},
\end{center}
where $v_{i_1}^{j_1}\neq v_{i_2}^{j_2}$.
Let the new digraph obtained through the above operations be $G_1$. Then
$$d_{G_1}^+(v_{i_1}^{j_1})=d_{F^{\frac{n+1}{2}}}^+(v_{i_1}^{j_1})-1 \ \textnormal{and}\ d_{G_1}^+(v_{i_2}^{j_2})=d_{F^{\frac{n+1}{2}}}^+(v_{i_2}^{j_2})+1.$$
At this point, there might be $M_1(F^{\frac{n+1}{2}})< M_1(G_1)$ or $G_1$ is not a $\overrightarrow{C_3}$-free digraph.
We denote the $\overrightarrow{C_3}$ by $u\rightarrow v\rightarrow w\rightarrow u$.
Next, we prove $M_1(F^{\frac{n+1}{2}})>M_1(G_1)$ if $G_1$ is a $\overrightarrow{C_3}$-free digraph. We consider it by the following two cases.

\medskip\noindent
\textbf{Case 1.} $i_1<i_2$.

If $i_1<i_2$, then $d_{F^{\frac{n+1}{2}}}^+(v_{i_1}^{j_1})>d_{F^{\frac{n+1}{2}}}^+(v_{i_2}^{j_2})$. The outdegree descending sequence of $F^{\frac{n+1}{2}}$ and $G_1$ satisfy the majorization. Regardless of whether $G_1$ is a $\overrightarrow{C_3}$-free digraph, we always have $M_1(F^{\frac{n+1}{2}})>M_1(G_1)$.

\medskip\noindent
\textbf{Case 2.} $i_1 \geq i_2$.

If $i_1 \geq i_2$, then $d_{F^{\frac{n+1}{2}}}^+(v_{i_1}^{j_1}) \leq d_{F^{\frac{n+1}{2}}}^+(v_{i_2}^{j_2})$. The outdegree descending sequence of $F^{\frac{n+1}{2}}$ and $G_1$ do not satisfy the majorization. Then there might exist a $G_1$ such that $M_1(F^{\frac{n+1}{2}})<M_1(G_1)$. We consider two cases: $i_1>i_2$ and $i_1=i_2$.

\medskip\noindent
\textbf{Case 2.1.} $i_1>i_2$.

If $i_1>i_2$, without loss of generality, let the new arc be $(v_{i_2}^{j_2},v_{i_2'}^{j_2'})$. Because $v_{i_2}^1$ is equivalent to $v_{i_2}^2$, we set $j_2=1$ and $j_2'=1$. From the structure of $F^{\frac{n+1}{2}}$, we know $i_2'<i_2$. Then we can find two $\overrightarrow{C_3}$ in $G_1$ (see Figure~\ref{fi:Case2.1}):
$$v_{i_2}^1 \rightarrow v_{i_2'}^1 \rightarrow v_{i_2'}^2 \rightarrow v_{i_2}^1\ \textnormal{or}\
v_{i_2}^1 \rightarrow v_{i_2'}^1 \rightarrow v_{i_2}^2 \rightarrow v_{i_2}^1.$$

\begin{figure}[htbp]
\centering
\includegraphics[scale=1.2]{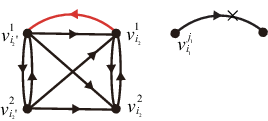}
\caption{The two $\protect\overrightarrow{C_3}$ in $G_1$ of Case 2.1}\label{fi:Case2.1}
\end{figure}

\medskip\noindent
\textbf{Case 2.2.} $i_1=i_2$.

If $i_1=i_2$, without loss of generality, let $v_{i_1}^{j_1}=v_{i_2}^2$, $v_{i_2}^{j_2}=v_{i_2}^1$ and delete the arc $(v_{i_2}^2,v_{i_1'}^{j_1'})$, add the arc $(v_{i_2}^1,v_{i_2'}^1)$. From the structure of $F^{\frac{n+1}{2}}$, we know $i_2'<i_2=i_1\leq i_1'$.

\medskip\noindent
\textbf{Case 2.2.1.} If $i_2=i_1<i_1'$, then we delete the arc $(v_{i_2}^2,v_{i_1'}^{j_1'})$ and add the arc $(v_{i_2}^1,v_{i_2'}^1)$. We also can find two $\overrightarrow{C_3}$ in $G_1$ (see Figure~\ref{fi:Case2.1.1}):
$$v_{i_2}^1 \rightarrow v_{i_2'}^1 \rightarrow v_{i_2'}^2 \rightarrow v_{i_2}^1\ \textnormal{or}\
v_{i_2}^1 \rightarrow v_{i_2'}^1 \rightarrow v_{i_2}^2 \rightarrow v_{i_2}^1.$$

\begin{figure}[htbp]
\centering
\includegraphics[scale=1.2]{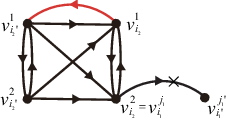}
\caption{The two $\protect\overrightarrow{C_3}$ in $G_1$ of Case 2.1.1}\label{fi:Case2.1.1}
\end{figure}

\medskip\noindent
\textbf{Case 2.2.2.} If $i_2=i_1=i_1'$, then $v_{i_1'}^{j_1'}=v_{i_2}^1$. That is, the deleted arc is $(v_{i_2}^2,v_{i_2}^1)$. We also can find $\overrightarrow{C_3}$ in $G_1$ (see Figure~\ref{fi:Case2.1.2}):
$$v_{i_2}^1 \rightarrow v_{i_2'}^1 \rightarrow v_{i_2'}^2 \rightarrow v_{i_2}^1.$$

\begin{figure}[htbp]
\centering
\includegraphics[scale=1.2]{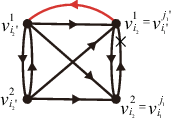}
\caption{The $\protect\overrightarrow{C_3}$ in $G_1$ of Case 2.1.2}\label{fi:Case2.1.2}
\end{figure}

From Case 2, if $i_1 \geq i_2$, we will all find a $\overrightarrow{C_3}$ in $G_1$: $v_{i_2}^1 \rightarrow v_{i_2'}^1 \rightarrow v_{i_2'}^2 \rightarrow v_{i_2}^1$.
So by Cases 1-2, we get $M_1(F^{\frac{n+1}{2}})>M_1(G_1)$ if $G_1$ is a $\overrightarrow{C_3}$-free digraph.

\medskip\medskip
We know that $G_1$ is obtained from $F^{\frac{n+1}{2}}$ by changing one arc. Next, we study the situation where multiple arcs are changed in $F^{\frac{n+1}{2}}$.

We will perform arc transformations based on $G_1$. From Case 1, we have $M_1(F^{\frac{n+1}{2}})>M_1(G_1)$; from Case 2, we might have $M_1(F^{\frac{n+1}{2}})<M_1(G_1)$ but $G_1$ is not a $\overrightarrow{C_3}$-free digraph. In order to find a $\overrightarrow{C_3}$-free digraph $G$ such that $M_1(F^{\frac{n+1}{2}})<M_1(G)$, we need to make $G_1$ in Case 2 has not $\overrightarrow{C_3}$ through arc transformations. We obtain the new digraph $G_2$ by deleting arcs and adding arcs in $G_1$. We take Case 2.1 as an example for analysis. The other cases are similar.

In Case 2.1, we know that $G_1$ has $\overrightarrow{C_3}$: $v_{i_2}^1 \rightarrow v_{i_2'}^1 \rightarrow v_{i_2'}^2 \rightarrow v_{i_2}^1$ or $v_{i_2}^1 \rightarrow v_{i_2'}^1 \rightarrow v_{i_2}^2 \rightarrow v_{i_2}^1$. If we delete the arc $(v_{i_2}^1,v_{i_2'}^1)$, then we also need add an arc and the cases similar to Cases 1-2. So we delete either $(v_{i_2'}^1, v_{i_2'}^2)$ or $(v_{i_2'}^2, v_{i_2}^1)$, and delete either $(v_{i_2'}^1, v_{i_2}^2)$ or $(v_{i_2}^2, v_{i_2}^1)$, such that $G_1$ has not $\overrightarrow{C_3}$. In order to make $e(G_1)=e(G_2)$, we add two arcs. Without loss of generality, assume that adding an arc with a tail of $v_{i_3}^{j_3}$ and an arc with a tail of $v_{i_4}^{j_4}$, where $i_4\leq i_3$. Next we are considering which arcs to delete.

If we delete either $(v_{i_2'}^1, v_{i_2'}^2)$ or $(v_{i_2'}^2, v_{i_2}^1)$, and delete $(v_{i_2'}^1, v_{i_2}^2)$, we consider two cases as follows.

\medskip\noindent
\textbf{Case 1'.} $i_2'<i_3$ and $i_2'<i_4$.

If $i_2'<i_3$ and $i_2'<i_4$, the outdegree descending sequence of $F^{\frac{n+1}{2}}$ and $G_2$ also satisfy the majorization and $M_1(F^{\frac{n+1}{2}})>M_1(G_2)$.

\medskip\noindent
\textbf{Case 2'.} $i_2'\geq i_3$ or $i_2'\geq i_4$.

If $i_2'\geq i_3$ or $i_2'\geq i_4$, then by the structure of $F^{\frac{n+1}{2}}$, we can find $\overrightarrow{C_3}$ in $G_2$ (see Figure~\ref{fi:Case22} for example):
$$v_{i_3}^{j_3} \rightarrow v_{i_3'}^1 \rightarrow v_{i_3'}^2 \rightarrow v_{i_3}^{j_3}$$
or
$$v_{i_4}^{j_4} \rightarrow v_{i_4'}^{2} \rightarrow v_{i_4'}^{1} \rightarrow v_{i_4}^{j_4},$$
where $i_3'<i_3$ and $i_4'<i_4$.

\begin{figure}[htbp]
\centering
\includegraphics[scale=1.2]{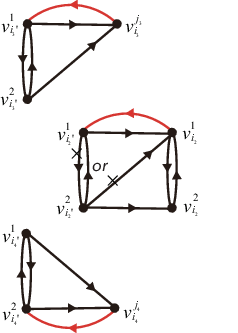}
\caption{The example of $\protect\overrightarrow{C_3}$ in $G_2$ of Case 2'}\label{fi:Case22}
\end{figure}

If we delete either $(v_{i_2'}^1, v_{i_2'}^2)$ or $(v_{i_2'}^2, v_{i_2}^1)$, and delete $(v_{i_2}^2, v_{i_2}^1)$, we consider two cases as follows.

\medskip\noindent
\textbf{Case 1''.} $i_3>i_2'$ and $i_4>i_2$.

If $i_3>i_2'$ and $i_4>i_2$, the outdegree descending sequence of $F^{\frac{n+1}{2}}$ and $G_2$ also satisfy the majorization and $M_1(F^{\frac{n+1}{2}})>M_1(G_2)$.

\medskip\noindent
\textbf{Case 2''.} $i_3\leq i_2'$ or $i_4\leq i_2$.

If $i_3\leq i_2'$ or $i_4\leq i_2$, then by the structure of $F^{\frac{n+1}{2}}$, we can find $\overrightarrow{C_3}$ in $G_2$ (see Figure~\ref{fi:Case222} for example):
$$v_{i_3}^{j_3} \rightarrow v_{i_3'}^1 \rightarrow v_{i_3'}^2 \rightarrow v_{i_3}^{j_3}$$
or
$$v_{i_4}^{j_4} \rightarrow v_{i_4'}^{2} \rightarrow v_{i_4'}^{1} \rightarrow v_{i_4}^{j_4},$$
where $i_3'<i_3$ and $i_4'<i_4$.

\begin{figure}[htbp]
\centering
\includegraphics[scale=1.2]{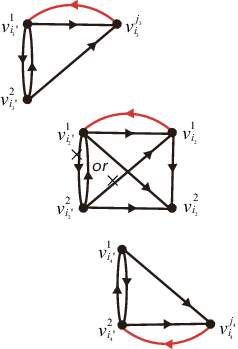}
\caption{The example of $\protect\overrightarrow{C_3}$ in $G_2$ of Case 2'}\label{fi:Case222}
\end{figure}

So by Cases 1' and 1'', we get $M_1(F^{\frac{n+1}{2}})>M_1(G_2)$. And by Cases 2' and 2'', we might have $M_1(F^{\frac{n+1}{2}})<M_1(G_2)$ but $G_2$ is not a $\overrightarrow{C_3}$-free digraph.

\medskip\medskip
We continue to perform arc transformations based on $G_2$ in Case 2' or Case 2'', to make $G_2$ has not $\overrightarrow{C_3}$. We also obtain the new digraph by deleting arcs and adding arcs in $G_2$ to keep the number of arcs unchanged. Then we are once again returning to the case similar to Cases 1-2 and starting the loop anew. In this loop, there has always been $G'$, such that $M_1(F^{\frac{n+1}{2}})>M_1(G')$ or $G'$ has $\overrightarrow{C_3}$.

Summing up the above, Claim~\ref{cl:ch-1} is completed. Hence, if $e(G^\ast)=\textnormal{ex}(n,\overrightarrow{C_3})$ and $r=1$, $\overrightarrow{{F}_{n,2}^{\frac{n+1}{2}}}$ is a $\overrightarrow{C_3}$-free digraph having the maximum First Zagreb Index.
If $e(G^\ast)=\textnormal{ex}(n,\overrightarrow{C_3})$ and $r=0$, the proof is similar to the case of $r=1$. Thus $\overrightarrow{{F}_{n,2}^{0}}$ is a $\overrightarrow{C_3}$-free digraph having the maximum First Zagreb Index.

\medskip\medskip
\noindent(ii) $e(G^\ast)<\textnormal{ex}(n,\overrightarrow{C_3})$.

If $e(G^\ast)<\textnormal{ex}(n,\overrightarrow{C_3})$, we also only consider the case when $r=1$; the case when $r=0$ is similar. We still perform the arc transformation on $F^{\frac{n+1}{2}}$.

First, we get a new digraph $F'$ by deleting some arcs in $F^{\frac{n+1}{2}}$, such that $e(F')=e(G^\ast)$.

Recall that $\mathcal{V}(F^{\frac{n+1}{2}})=\{v_1^1,v_1^2,v_2^1,v_2^2,\ldots,v_{\frac{n-1}{2}}^1,v_{\frac{n-1}{2}}^2,v_{\frac{n+1}{2}}^1\}$ and the outdegree sequence of $F^{\frac{n+1}{2}}$ is $\{n-1,n-1,n-3,n-3,n-5,n-5,\ldots,4,4,2,2,0\}$.

Let $e'=\textnormal{ex}(n,\overrightarrow{C_3})-e(G^\ast)$. In $F^{\frac{n+1}{2}}$, we start deleting arcs from the vertex with the smallest outdegree, as shown in the table~\ref{ta:ch-1}.

\begin{table}[ht]
\renewcommand\arraystretch{1.5}\centering
\caption{The impact of the number of arcs deleted from $F^{\frac{n+1}{2}}$ on the outdegree of $F'$}\label{ta:ch-1}
\begin{tabular}{|c|c|}
\hline
$e'$ & The outdegree of $F'$  \\ \hline
1 & \tabincell{c}{$d_{F'}(v)=d_{F^{\frac{n+1}{2}}}(v)$ for $v\in\mathcal{V}(F^{\frac{n+1}{2}})$ and $v\neq v_{\frac{n-1}{2}}^2$;\\ $d_{F'}(v_{\frac{n-1}{2}}^2)=1$}  \\ \hline
2 & \tabincell{c}{$d_{F'}(v)=d_{F^{\frac{n+1}{2}}}(v)$ for $v\in\mathcal{V}(F^{\frac{n+1}{2}})$ and $v\neq v_{\frac{n-1}{2}}^2$;\\ $d_{F'}(v_{\frac{n-1}{2}}^2)=0$}  \\ \hline
3 & \tabincell{c}{$d_{F'}(v)=d_{F^{\frac{n+1}{2}}}(v)$ for $v\in\mathcal{V}(F^{\frac{n+1}{2}})$ and $v\neq v_{\frac{n-1}{2}}^1,v_{\frac{n-1}{2}}^2$;\\ $d_{F'}(v_{\frac{n-1}{2}}^1)=1$ and $d_{F'}(v_{\frac{n-1}{2}}^2)=0$}  \\ \hline
4 & \tabincell{c}{$d_{F'}(v)=d_{F^{\frac{n+1}{2}}}(v)$ for $v\in\mathcal{V}(F^{\frac{n+1}{2}})$ and $v\neq v_{\frac{n-1}{2}}^1,v_{\frac{n-1}{2}}^2$;\\ $d_{F'}(v_{\frac{n-1}{2}}^1)=0$ and $d_{F'}(v_{\frac{n-1}{2}}^2)=0$}  \\ \hline
$\vdots$ & $\vdots$ \\ \hline
\end{tabular}
\end{table}

For the new digraph $F'$, we know $F'$ is a $\overrightarrow{C_3}$-free digraph and $M_1(F^{\frac{n+1}{2}})>M_1(F')$. We then increase or decrease the arcs in $F'$ to obtain the new digraph $G'$, where $e(F')=e(G')$. Similar to the aforementioned proof of Claim~\ref{cl:ch-1}, if the outdegree descending sequence of $F'$ and $G'$ satisfy the majorization, then $M_1(F')>M_1(G')$; if the outdegree descending sequence of $F'$ and $G'$ do not satisfy the majorization, then $G'$ has $\overrightarrow{C_3}$. So, we can not find a $\overrightarrow{C_3}$-free digraph $G$, such that $M_1(F')\leq M_1(G)$.

So if  $e(G^\ast)<\textnormal{ex}(n,\overrightarrow{C_3})$, we get $M_1(F^{\frac{n+1}{2}})>M_1(G^\ast)$, a contradiction.

\medskip\medskip
Therefore, we obtain
$$\textnormal{ex}_{M_1}(n,\overrightarrow{C_3})=\sum_{i=1}^q 2(n-1-2(i-1))^2=\frac{2q}{3}(3n^2-6qn+4q^2-1),$$
and
$$\textnormal{EX}_{M_1}(n,\overrightarrow{C_3})=
\begin{cases}
\overrightarrow{{F}_{n,2}^{\frac{n+1}{2}}},& \textnormal{if} \ r=1,\\
\overrightarrow{{F}_{n,2}^0},& \textnormal{if} \ r=0.
\end{cases}$$

\end{proof}

Next, we consider the extremal Laplacian energy of $\overrightarrow{C_3}$-free digraphs.

\medskip\medskip
\noindent\textbf{Proof of Theorem~\ref{th:exLEC3}:}

Since Laplacian energy $LE(G)=\sum_{i=1}^n(d^+_i)^2+c_2=M_1+c_2$, we need to consider $c_2$. Assume that $G^\ast$ is a $\overrightarrow{C_3}$-free digraph having the maximum Laplacian energy and consider two cases: (i) $e(G^\ast)=\textnormal{ex}(n,\overrightarrow{C_3})$; (ii) $e(G^\ast)<\textnormal{ex}(n,\overrightarrow{C_3})$. We continue the proof based on the proof of Lemma~\ref{le:exMC3}. From Lemma~\ref{le:exMC3}, we have obtained $\textnormal{EX}_{M_1}(n,\overrightarrow{C_3})=\overrightarrow{{F}_{n,2}^{\frac{n+1}{2}}}$ if $r=1$ and
$\textnormal{EX}_{M_1}(n,\overrightarrow{C_3})=\overrightarrow{{F}_{n,2}^0}$ if $r=0$.
We also only consider the case when $r=1$; the case when $r=0$ is similar.

\medskip\medskip
\noindent (i) $e(G^\ast)=\textnormal{ex}(n,\overrightarrow{C_3})$.

By the Claim~\ref{cl:ch-1} of Lemma~\ref{le:exMC3}, we have already obtained $F^{\frac{n+1}{2}}$ is a $\overrightarrow{C_3}$-free digraph having the maximum First Zagreb Index. By the operations on the arcs in $F^{\frac{n+1}{2}}$:
\begin{center}
\textbf{delete an arc with a tail of $v_{i_1}^{j_1}$ and add an arc with a tail of $v_{i_2}^{j_2}$}.
\end{center}
Adding $(v_{i_2}^{j_2},v_{i'_2}^{j'_2})$ must add one directed closed walk of length 2; deleting $(v_{i_1}^{j_1},v_{i'_1}^{j'_1})$, $c_2(F^{\frac{n+1}{2}})$ may remain unchanged or decrease by 2, where $i'_2<i_2$ and $i_1\leq i'_1$. So $c_2(G_1)=c_2(F^{\frac{n+1}{2}})+2$ or $c_2(G_1)=c_2(F^{\frac{n+1}{2}})$.

By the Case~2 ($i_1\geq i_2$) in Claim~\ref{cl:ch-1}, we will always find a $\overrightarrow{C_3}$ in the new digraph. By the Case~1 ($i_1<i_2$) in Claim~\ref{cl:ch-1}, since $d_{G_1}^+(v_{i_1}^{j_1})=d_{F^{\frac{n+1}{2}}}^+(v_{i_1}^{j_1})-1$, $d_{G_1}^+(v_{i_2}^{j_2})=d_{F^{\frac{n+1}{2}}}^+(v_{i_2}^{j_2})+1$, and $d_{F^{\frac{n+1}{2}}}^+(v_{i_1}^{j_1})>d_{F^{\frac{n+1}{2}}}^+(v_{i_2}^{j_2})$, we have
$$M_1(G_1)-\left(d_{F^{\frac{n+1}{2}}}^+(v_{i_1}^{j_1})-1\right)^2-\left(d_{F^{\frac{n+1}{2}}}^+(v_{i_2}^{j_2})+1\right)^2=
M_1(F^{\frac{n+1}{2}})-\left(d_{F^{\frac{n+1}{2}}}^+(v_{i_1}^{j_1})\right)^2-\left(d_{F^{\frac{n+1}{2}}}^+(v_{i_2}^{j_2})\right)^2.$$

So
\begin{align*}
LE(G_1)&=M_1(G_1)+c_2(G_1)\\
&\leq M_1(F^{\frac{n+1}{2}})-\left(d_{F^{\frac{n+1}{2}}}^+(v_{i_1}^{j_1})\right)^2-\left(d_{F^{\frac{n+1}{2}}}^+(v_{i_2}^{j_2})\right)^2\\
&+\left(d_{F^{\frac{n+1}{2}}}^+(v_{i_1}^{j_1})-1\right)^2+\left(d_{F^{\frac{n+1}{2}}}^+(v_{i_2}^{j_2})+1\right)^2+c_2(F^{\frac{n+1}{2}})+2\\
&=LE(F^{\frac{n+1}{2}})-2\left(d_{F^{\frac{n+1}{2}}}^+(v_{i_1}^{j_1})-d_{F^{\frac{n+1}{2}}}^+(v_{i_2}^{j_2})\right)+4\\
&\leq LE(F^{\frac{n+1}{2}})-4+4=LE(F^{\frac{n+1}{2}}).
\end{align*}
The first inequality is an equation if and only if
$$c_2(G_1)=c_2(F^{\frac{n+1}{2}})+2.$$
Thus, $i_1<i'_1$. The second inequality is an equation if and only if
$$d_{F^{\frac{n+1}{2}}}^+(v_{i_1}^{j_1})-d_{F^{\frac{n+1}{2}}}^+(v_{i_2}^{j_2})=2,$$
by this time, $i_1+1=i_2$. But $G_1$ must has $\overrightarrow{C_3}$: $v_{i_2}^1\rightarrow v_{i'_2}^{j'_2}\rightarrow v_{i_2}^2\rightarrow v_{i_2}^1$ (There might be other $\overrightarrow{C_3}$ as well). So merely changing one arc is not enough. We need to perform multiple arc transformations.

Let $G_s$ denote the new digraph obtained after each arc transformation. Suppose that there are $t$ arcs undergoing arc transformations such that the final digraph $G_t$ is a $\overrightarrow{C_3}$-free digraph and $G_s$ is not a $\overrightarrow{C_3}$-free digraph, where $s=1,2,\ldots,t-1$. Let us assume that the arcs that undergoes a deletion or addition operation each time be: delete arc $(v_{i_{2s-1}}^{j_{2s-1}},v_{i'_{2s-1}}^{j'_{2s-1}})$ and add arc $(v_{i_{2s}}^{j_{2s}},v_{i'_{2s}}^{j'_{2s}})$, where $s=1,2,\ldots,t$. (Note: these vertices may overlap.)

Let $x=\text{min}\{i_{2s-1},i'_{2s-1},i_{2s},i'_{2s}\}$ and $y=\text{max}\{i_{2s-1},i'_{2s-1},i_{2s},i'_{2s}\}$ be the minimum and maximum indices of these vertices, respectively, where $s=1,2,\ldots,t$. We place all the vertices labeled as $\{x, x+1, x+2,\ldots, y-1,y\}$ into one vertex set $U$. That is, we place all the vertices in $\left\{F^{\frac{n+1}{2}}[V_i]: i=x,x+1,x+2,\ldots,y-1,y\right\}$ into one vertex set $U$.

Actually, we know that $\textnormal{EX}(n,\overrightarrow{C_3})=\overrightarrow{\mathcal{BK}_{n,p}}$, where $\overrightarrow{BK_{n,p}}[V_i]$ is a balanced complete bipartite digraph for $i\in[1,p]$ and $|V_i|=n_i$ is an even for all but at most one $V_i$. So the induced subdigraph by the vertices in $U$ must be a balanced complete bipartite digraph.

Through the proof of Lemma~\ref{le:exMC3} and the analysis of a single arc transformation, we get $LE(G_1)\leq LE(F^{\frac{n+1}{2}})$. In order to find the $\overrightarrow{C_3}$-free digraph with maximum Laplacian energy, next, we consider the number of vertices of $U$ such that $LE(G_t)=LE(F^{\frac{n+1}{2}})$.

We divide the vertex set of $F^{\frac{n+1}{2}}$ and $G_t$ into three subsets: $X=\bigcup_{i=1}^{x-1}V_i$, $U=\bigcup_{i=x}^{y}V_i$, $Y=\bigcup_{i=y+1}^{q'}V_i$. Since the arcs in $X$ and $Y$, the arcs between $X$ and $Y$, the arcs between $X$ and $U$ and the arcs between $U$ and $Y$ remain unchanged, we only need consider the outdegrees of the vertices and the directed closed walk of length 2 in $U$. Let $|U|=\ell$.

\medskip\noindent\textbf{Case 1.} If $|U|$ is odd, then $Y=\emptyset$.

For $U$ in $\overrightarrow{F^{\frac{n+1}{2}}}$, the sum of the squares of the outdegrees of the vertices plus the directed closed walk of length 2 is
$$\sum_{i=1}^{(\ell-1)/2}2(\ell-1-2(i-1))^2+\ell-1=\frac{1}{3}{\ell}^3+\frac{2}{3}\ell-1.$$

For $U$ in $G_t$, since $\overrightarrow{G_t}[U]$ is a balanced complete bipartite digraph, the sum of the squares of the outdegrees of the vertices plus the directed closed walk of length 2 is
$$\frac{\ell+1}{2}\left(\frac{\ell-1}{2}\right)^2+\frac{\ell-1}{2}\left(\frac{\ell+1}{2}\right)^2
+2\left(\frac{\ell-1}{2}\right)\left(\frac{\ell+1}{2}\right)
=\frac{1}{4}{\ell}^3+\frac{1}{2}{\ell}^2-\frac{1}{4}{\ell}-\frac{1}{2}.$$

Let $$\frac{1}{3}{\ell}^3+\frac{2}{3}\ell-1=\frac{1}{4}{\ell}^3+\frac{1}{2}{\ell}^2-\frac{1}{4}{\ell}-\frac{1}{2}.$$
Solve for $\ell$, we get $\ell=1,2,3$. For the assumption of $U$, $|U|=\ell=3$.

\medskip\noindent\textbf{Case 2.} If $|U|$ is even, then $Y\neq\emptyset$.

For $U$ in $\overrightarrow{F^{\frac{n+1}{2}}}$, the sum of the squares of the outdegrees of the vertices plus the directed closed walk of length 2 is
$$\sum_{i=1}^{\ell/2}2(\ell-1+n-2y-2(i-1))^2+\ell=\frac{1}{3}{\ell}^3+(n-2y){\ell}^2+\left((n-2y)^2+\frac{2}{3}\right)\ell.$$

For $U$ in $G_t$, since $\overrightarrow{G_t}[U]$ is a balanced complete bipartite digraph, the sum of the squares of the outdegrees of the vertices plus the directed closed walk of length 2 is
$$\ell\left(\frac{\ell}{2}+n-2y\right)^2+2\left(\frac{\ell}{2}\right)^2=\frac{1}{4}{\ell}^3+\left(n-2y+\frac{1}{2}\right){\ell}^2+(n-2y)^2\ell.$$

Let $$\frac{1}{3}{\ell}^3+(n-2y){\ell}^2+\left((n-2y)^2+\frac{2}{3}\right)\ell=\frac{1}{4}{\ell}^3+\left(n-2y+\frac{1}{2}\right){\ell}^2+(n-2y)^2\ell.$$
Solve for $\ell$,we get $\ell=0,2,4$. For the assumption of $U$, $|U|=\ell=4$.

Hence, by Case 1-2, we get $G_t\subseteq\overrightarrow{\mathcal{BK}^1_{n,p}}$ like that: for integer $n\in \mathbb{N}^+$, $n=2q+1$,
\begin{itemize}
  \item [(a)] $\mathcal{V}(G_t)$ has a partition $\{V_1, V_2, \ldots, V_{q-1}\}$ such that $|V_i|=2$ for all but at most one $|V_i|=4$ and $|V_q|=1$, or $|V_i|=2$ for all and $|V_{q}|=3$;
  \item [(b)] $\overrightarrow{F_{n,k}}[V_i]$ is a balanced complete bipartite digraph for $i\in[1,q]$;
  \item [(c)] for any two parts $V_i$ and $V_j$ with $1\leq i<j\leq q$, $V_i \mapsto V_j$.
\end{itemize}

Building upon $G_t=G^0_t$, we perform further arc transformations until the resulting digraph $G^1_t$ is $\overrightarrow{C_3}$-free. This process is analogous to the analysis of $G_t$: we collect all vertices whose indegree or outdegree changes, along with those vertices in the middle of these vertices that have a transmission relationship, into a single set $U^1$ such that $LE(G^1_t)=LE(G_t)$. Ultimately, we still obtain $|U^1|=3$ or $4$.
Building upon $G^1_t$, we perform further arc transformations until the resulting digraph $G^2_t$ is $\overrightarrow{C_3}$-free. Then we also can get $|U^2|=3$ or $4$ such that $LE(G^2_t)=LE(G^1_t)$. Proceeding in this way, we can always obtain a $\overrightarrow{C_3}$-free digraph until there do not exist two adjacent vertex sets each of order 2. Moreover, all digraphs $\{G^0_t,G^1_t,G^2_t,\cdots\}$ obtained during this process are $\overrightarrow{C_3}$-free digraph and
$$LE(F^{\frac{n+1}{2}})=LE(G^0_t)=LE(G^1_t)=LE(G^2_t)=\cdots.$$
Also, $\{F^{\frac{n+1}{2}},G^0_t,G^1_t,G^2_t,\cdots\}\subseteq\overrightarrow{\mathcal{BK}^1_{n,p}}$.

Therefore, $\overrightarrow{BK^1_{n,p}}\subseteq\overrightarrow{\mathcal{BK}^1_{n,p}}$ is a $\overrightarrow{C_3}$-free digraph having the maximum Laplacian energy if $e(G^\ast)=\textnormal{ex}(n,\overrightarrow{C_3})$.

\medskip\medskip
\noindent(ii) $e(G^\ast)<\textnormal{ex}(n,\overrightarrow{C_3})$.

From the proof of Lemma~\ref{le:exMC3}, we have obtained $F'$ is a $\overrightarrow{C_3}$-free digraph and $M_1(F^{\frac{n+1}{2}})>M_1(F')>M_1(G')$, where $G'$ is an any $\overrightarrow{C_3}$-free digraph such that $e(F^{\frac{n+1}{2}})>e(F')=e(G')$ and $G'\ncong F'$. Similar to the aforementioned proof of $c_2$ in $F^{\frac{n+1}{2}}$ and $G_1$ when $e(G^\ast)=\textnormal{ex}(n,\overrightarrow{C_3})$, we can get
$$LE(F')\geq LE(G').$$
Since
\begin{align*}
&LE(\overrightarrow{BK^1_{n,p}})=LE(F^{\frac{n+1}{2}})=M_1(F^{\frac{n+1}{2}})+c_2(F^{\frac{n+1}{2}})\\
&>M_1(F')+c_2(F')=LE(F')\geq LE(G'),
\end{align*}
we get $\overrightarrow{BK^1_{n,p}}\subseteq\overrightarrow{\mathcal{BK}^1_{n,p}}$ is a $\overrightarrow{C_3}$-free digraph having the maximum Laplacian energy.

\medskip\medskip
Thus, $$\textnormal{EX}_{LE}(n,\overrightarrow{C_3})=\overrightarrow{\mathcal{BK}^1_{n,p}}$$ when $r=1$.

\medskip\medskip
Similarly, $$\textnormal{EX}_{LE}(n,\overrightarrow{C_3})=\overrightarrow{\mathcal{BK}^0_{n,p}}$$ when $r=0$.

\medskip\medskip
In conclusion, we obtain
$$\textnormal{ex}_{LE}(n,\overrightarrow{C_3})=\sum_{i=1}^q 2(n-1-2(i-1))^2+2q=\frac{2q}{3}(3n^2-6qn+4q^2+2),$$
and
$$\textnormal{EX}_{LE}(n,\overrightarrow{C_3})=
\begin{cases}
\overrightarrow{\mathcal{BK}^1_{n,p}},& \textnormal{if} \ r=1,\\
\overrightarrow{\mathcal{BK}^0_{n,p}},& \textnormal{if} \ r=0.
\end{cases}$$
This completes the proof of Theorem~\ref{th:exLEC3}.

$\hfill\square$

\medskip\medskip
As an illustration of Theorem~\ref{th:exLEC3}, Figure~\ref{fi:example1} shows all $\overrightarrow{C_3}$-free digraphs of order $11$ having the maximum Laplacian energy.

\begin{figure}[htbp]
    \centering
    \begin{subfigure}[b]{\textwidth}
        \includegraphics[scale=0.7]{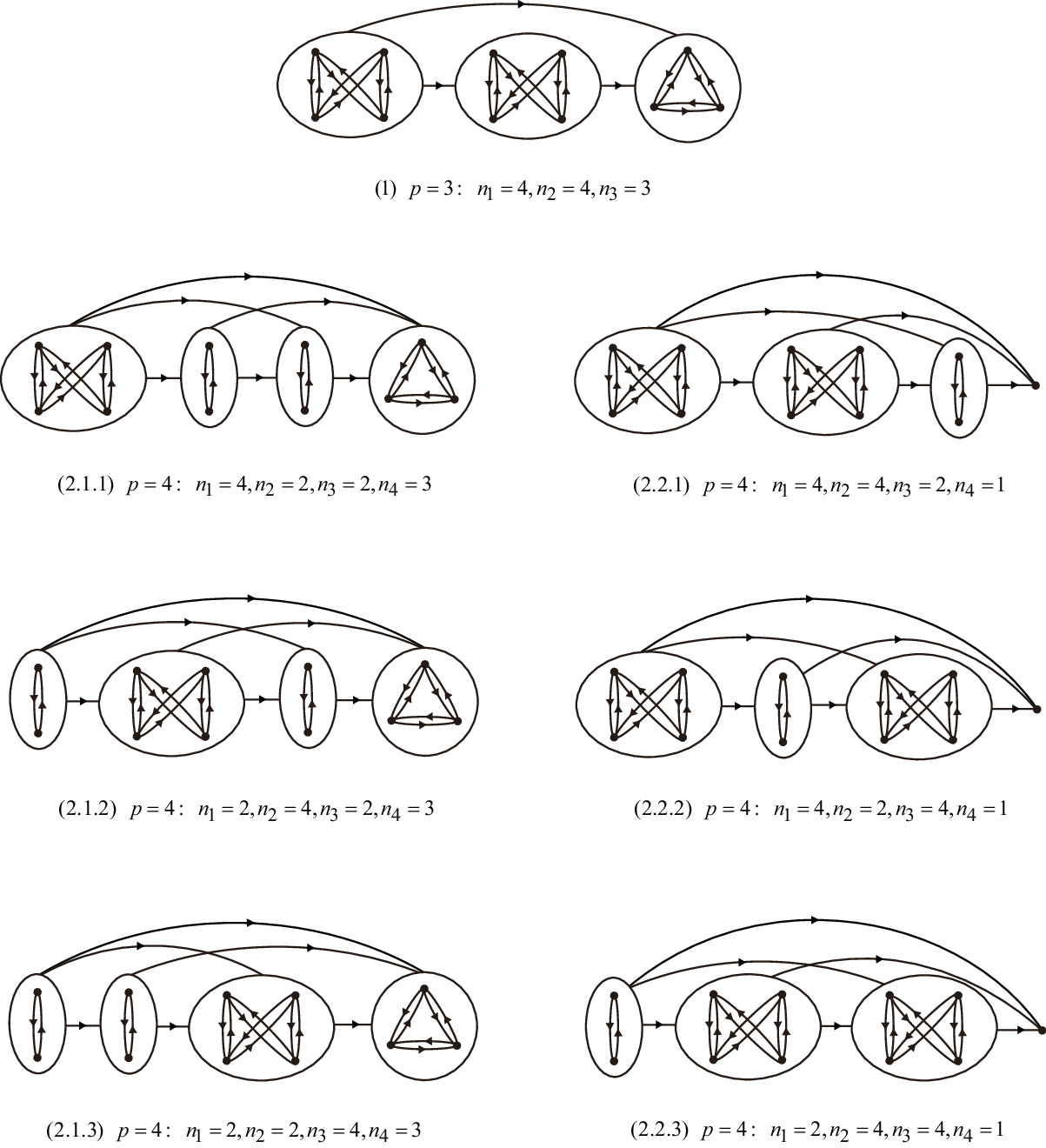}
    \end{subfigure}
    \caption{$\textnormal{EX}_{LE}(11,\protect\overrightarrow{C_3})$}
    \label{fi:example1}
\end{figure}

\begin{figure}[htbp]
\ContinuedFloat
    \setcounter{subfigure}{2}
    \centering
    \begin{subfigure}[b]{\textwidth}
        \includegraphics[scale=0.7]{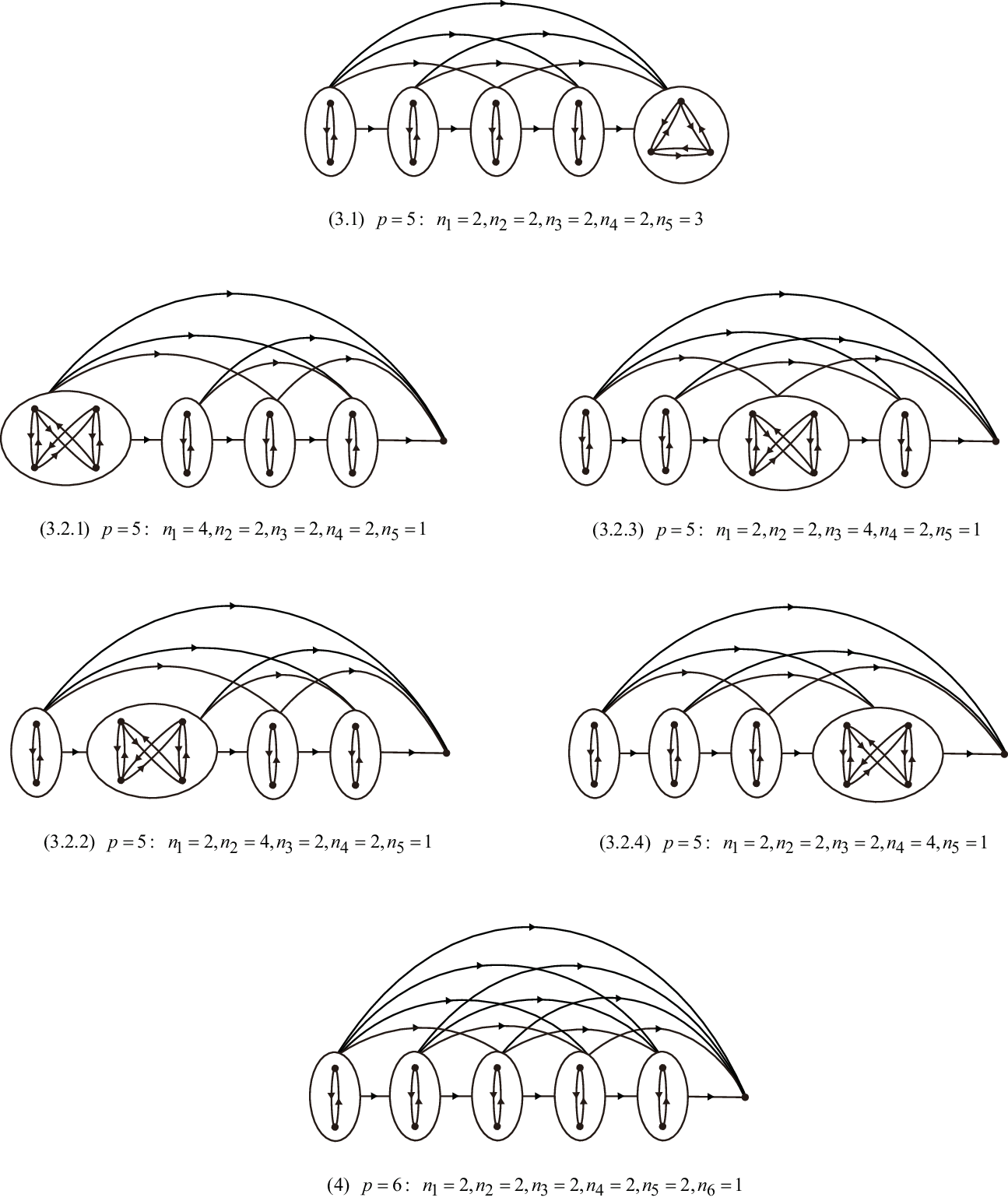}
    \end{subfigure}
    %\caption{$\textnormal{EX}_{LE}(11,\protect\overrightarrow{C_3})$} 
    \label{fi:example2}
\end{figure}

{\section{Extremal Laplacian energy of $\protect\overrightarrow{C_{k+1}}$-free digraph when $k\geq3$}\label{sec3}}

In this section, we give the proof of Theorem~\ref{th:exLEC}.
By Theorem~\ref{th:exC}, we have $\textnormal{EX}(n,\overrightarrow{C_{k+1}})=\overrightarrow{\mathcal{F}_{n,k}}$ when $k\geq3$. We first consider the digraph which has the maximum Laplacian energy in $\overrightarrow{\mathcal{F}_{n,k}}$. In fact, we can get the ordering of the Laplacian energy of digraphs in $\overrightarrow{\mathcal{F}_{n,k}}$ when $\overleftrightarrow{K_r}$ at different locations.

\noindent\begin{lemma}\label{le:ordering} Let $k,n\in \mathbb{N}^+$, $n=qk+r$, $0<r<k$. Let $\overrightarrow{{F}_{n,k}^{s+1}}\in\overrightarrow{\mathcal{F}_{n,k}}$ such that $\overrightarrow{F_{n,k}}[V_{s+1}]=\overleftrightarrow{K_r}$ and $\overrightarrow{F_{n,k}}[V_{i}]=\overleftrightarrow{K_k}$ for $i=1,\ldots,s,s+2,\ldots,q+1$, where $s=0,1,\ldots,q$. Then
$$LE(\overrightarrow{{F}_{n,k}^1})<LE(\overrightarrow{{F}_{n,k}^2})<\cdots<
LE(\overrightarrow{{F}_{n,k}^{q}})<LE(\overrightarrow{{F}_{n,k}^{q+1}}).$$
\end{lemma}
\begin{proof} Let $\overrightarrow{{F}_{n,k}^{s+1}}\in\overrightarrow{\mathcal{F}_{n,k}}$ such that $\overrightarrow{F_{n,k}}[V_{s+1}]=\overleftrightarrow{K_r}$ and $\overrightarrow{F_{n,k}}[V_{i}]=\overleftrightarrow{K_k}$ for $i=1,\ldots,s,s+2,\ldots,q+1$, where $s=0,1,\ldots,q$. That is,
$$\begin{cases}
\overrightarrow{F_{n,k}}[V_{1}]=\overrightarrow{F_{n,k}}[V_{2}]=\cdots=\overrightarrow{F_{n,k}}[V_{s}]=\overleftrightarrow{K_k},\\
\overrightarrow{F_{n,k}}[V_{s+1}]=\overleftrightarrow{K_r},\\
\overrightarrow{F_{n,k}}[V_{s+2}]=\cdots=\overrightarrow{F_{n,k}}[V_{q+1}]=\overleftrightarrow{K_k}.
\end{cases}$$
Let $V_{i}=\{v_i^1,v_i^2,\ldots,v_i^k\}$, where $i=1,2,\ldots,s,s+2,\ldots,q+1$ and $V_{s+1}=\{v_{s+1}^1,v_{s+1}^2,\ldots,v_{s+1}^r\}$. Then
$$\begin{cases}
d_G^+(v_i^j)=n-1-(i-1)k, & i=1,2,\ldots,s,\ j=1,2,\ldots,k,\\
d_G^+(v_{s+1}^j)=n-1-sk, & j=1,2,\ldots,r,\\
d_G^+(v_i^j)=n-1-(i-2)k-r, & i=s+2,\ldots,q+1,\ j=1,2,\ldots,k.
\end{cases}$$
Thus, the outdegree descending sequence of $\overrightarrow{{F}_{n,k}^{s+1}}$ is
$$\left\{\underbrace{n-1}_{k},\ \underbrace{n-1-k}_{k}, \ldots, \underbrace{n-1-(s-1)k}_{k},\ \underbrace{n-1-sk}_{r},\
\underbrace{n-1-sk-r}_{k}, \ldots, \underbrace{n-1-(q-1)k-r}_{k}\right\}.$$

Let $Sd_t(\overrightarrow{{F}_{n,k}^{s+1}})$ be the sum of the $t$ largest outdegree of $\overrightarrow{{F}_{n,k}^{s+1}}$. For convenience, we denote $\overrightarrow{{F}_{n,k}^{s+1}}=F^{s+1}$. Next we prove $Sd_t(F^{s_2+1})\leq Sd_t(F^{s_1+1})$ for $s_2<s_1$. For $t$, we have five cases: $1\leq t\leq s_2 k$, $s_2 k+1\leq t\leq s_2 k+r$, $s_2 k+r+1\leq t\leq s_1 k$, $s_1 k+1\leq t\leq s_1 k+r$, $s_1 k+r+1\leq t\leq n$.

\medskip\noindent\textbf{Case 1.} If $1\leq t\leq s_2 k$, then $Sd_t(F^{s_2+1})=Sd_t(F^{s_1+1})$.

Since $1\leq t\leq s_2 k< s_1 k$, we have $F^{s_2+1}[V_i]=F^{s_1+1}[V_i]=\overleftrightarrow{K_k}$ for $i=1,2,\ldots,s_2$.
Let $t=q_1k+t_1$, where $q_1\leq s_2$ and $0\leq t_1<k$. Then
$$Sd_t(F^{s_2+1})=Sd_t(F^{s_1+1})=\sum_{i=1}^{q_1} k(n-1-(i-1)k)+t_1(n-1-q_1 k).$$

\medskip\noindent\textbf{Case 2.} If $s_2 k+1\leq t\leq s_2 k+r$, then $Sd_t(F^{s_2+1})=Sd_t(F^{s_1+1})$.

Since $s_2 k+1\leq t\leq s_2 k+r$, we have $F^{s_2+1}[V_i]=F^{s_1+1}[V_i]=\overleftrightarrow{K_k}$ for $i=1,2,\ldots,s_2$, $F^{s_2+1}[V_{s_2+1}]=\overleftrightarrow{K_r}$ and $F^{s_1+1}[V_{s_2+1}]=\overleftrightarrow{K_k}$. So
$$Sd_t(F^{s_2+1})=Sd_t(F^{s_1+1})=\sum_{i=1}^{s_2} k(n-1-(i-1)k)+(t-s_2k)(n-1-s_2 k).$$

\medskip\noindent\textbf{Case 3.} If $s_2 k+r+1\leq t\leq s_1 k$, then $Sd_t(F^{s_2+1})<Sd_t(F^{s_1+1})$.

Since $s_2 k+r+1\leq t\leq s_1 k$, we have $F^{s_2+1}[V_i]=\overleftrightarrow{K_k}$ for $i=1,2,\ldots,s_2,s_2+2,\ldots,s_1$, $F^{s_2+1}[V_{s_2+1}]=\overleftrightarrow{K_r}$, and $F^{s_1+1}[V_i]=\overleftrightarrow{K_k}$ for $i=1,2,\ldots,s_1$.
For $Sd_t(F^{s_2+1})$, let $t=s_2k+r+q_2k+t_2$, where $0\leq t_2<k$ and $0<r<k$. Then
\begin{align*}
Sd_t(F^{s_2+1})&=\sum_{i=1}^{s_2} k(n-1-(i-1)k)+r(n-1-s_2k)\\
&+\sum_{i=s_2+2}^{s_2+q_2+1} k(n-1-(i-2)k-r)+t_2(n-1-(s_2+q_2)k-r).
\end{align*}
For $Sd_t(F^{s_1+1})$, since $t=(s_2+q_2)k+t_2+r$, we consider two cases: $0\leq t_2+r<k$ and $k\leq t_2+r<2k$. Then
\begin{align*}
Sd_t(F^{s_1+1})&=\sum_{i=1}^{s_2+q_2} k(n-1-(i-1)k)\\
&+\begin{cases}
(t_2+r)(n-1-(s_2+q_2)k), & 0\leq t_2+r<k,\\
k(n-1-(s_2+q_2)k)+(t_2+r-k)(n-1-(s_2+q_2+1)k), & k\leq t_2+r<2k.
\end{cases}
\end{align*}

\medskip\noindent\textbf{Case 3.1.} If $0\leq t_2+r<k$, then we get
\begin{align*}
&Sd_t(F^{s_1+1})-Sd_t(F^{s_2+1})\\
&=\sum_{i=s_2+1}^{s_2+q_2} k(n-1-(i-1)k)+(t_2+r)(n-1-(s_2+q_2)k)\\
&-\sum_{i=s_2+2}^{s_2+q_2+1} k(n-1-(i-2)k-r)-r(n-1-s_2k)-t_2(n-1-(s_2+q_2)k-r)\\
&=t_2(k+r)>0.
\end{align*}

\medskip\noindent\textbf{Case 3.2.} If $k\leq t_2+r<2k$, then we get
\begin{align*}
&Sd_t(F^{s_1+1})-Sd_t(F^{s_2+1})\\
&=\sum_{i=s_2+1}^{s_2+q_2} k(n-1-(i-1)k)+k(n-1-(s_2+q_2)k)+(t_2+r-k)(n-1-(s_2+q_2+1)k)\\
&-\sum_{i=s_2+2}^{s_2+q_2+1} k(n-1-(i-2)k-r)-r(n-1-s_2k)-t_2(n-1-(s_2+q_2)k-r)\\
&=k(k-r)+t_2r>0.
\end{align*}

Thus $Sd_t(F^{s_2+1})<Sd_t(F^{s_1+1})$.

\medskip\noindent\textbf{Case 4.} If $s_1 k+1\leq t\leq s_1 k+r$, then $Sd_t(F^{s_2+1})<Sd_t(F^{s_1+1})$.

Since $s_1 k+1\leq t\leq s_1 k+r$, we have $F^{s_2+1}[V_i]=\overleftrightarrow{K_k}$ for $i=1,2,\ldots,s_2,s_2+2,\ldots,s_1+1$, $F^{s_2+1}[V_{s_2+1}]=\overleftrightarrow{K_r}$, and $F^{s_1+1}[V_i]=\overleftrightarrow{K_k}$ for $i=1,2,\ldots,s_1$, $F^{s_1+1}[V_{s_1+1}]=\overleftrightarrow{K_r}$.
For $Sd_t(F^{s_1+1})$, let $t=s_1k+t_3$, where $0<t_3\leq r<k$. Then
$$Sd_t(F^{s_1+1})=\sum_{i=1}^{s_1} k(n-1-(i-1)k)+t_3(n-1-s_1k).$$
For $Sd_t(F^{s_2+1})$, since $t=s_2k+r+(s_1-s_2-1)k+(t_3+k-r)$, where $0<t_3+k-r<k$, then
\begin{align*}
Sd_t(F^{s_2+1})&=\sum_{i=1}^{s_2} k(n-1-(i-1)k)+r(n-1-s_2k)\\
&+\sum_{i=s_2+2}^{s_1} k(n-1-(i-2)k-r)+(t_3+k-r)(n-1-(s_1-1)k-r).
\end{align*}

So we get
\begin{align*}
&Sd_t(F^{s_1+1})-Sd_t(F^{s_2+1})\\
&=\sum_{i=s_2+1}^{s_1} k(n-1-(i-1)k)+t_3(n-1-s_1k)-r(n-1-s_2k)\\
&-\sum_{i=s_2+2}^{s_1} k(n-1-(i-2)k-r)-(t_3+k-r)(n-1-(s_1-1)k-r)\\
&=(k-r)(r-t_3)>0.
\end{align*}

Thus $Sd_t(F^{s_2+1})<Sd_t(F^{s_1+1})$.

\medskip\noindent\textbf{Case 5.} If $s_1 k+r+1\leq t\leq n$, then $Sd_t(F^{s_2+1})=Sd_t(F^{s_1+1})$.

We consider in reverse. Since $d_{F^{s_2+1}}^+(v_i^j)=n-1-(i-2)k-r$ and $d_{F^{s_1+1}}^+(v_i^j)=n-1-(i-2)k-r$, where $i=s_1+1,\ldots,q+1$ and $j=1,2,\ldots,k$, we get $Sd_t(F^{s_2+1})=Sd_t(F^{s_1+1})$.

\medskip
From Cases 1-5, we obtain $Sd_t(F^{s_2+1})\leq Sd_t(F^{s_1+1})$. Furthermore, $Sd_n(F^{s_2+1})=Sd_n(F^{s_1+1})$. That is, the outdegree descending sequence of $F^{s_1+1}$ majorizes $F^{s_2+1}$. Also, $d_{F^{s_2+1}}^+(v_i^j)\neq d_{F^{s_1+1}}^+(v_i^j)$ for all $v_i^j$. So from Karamata's inequality, we obtain $LE(F^{s_2+1})<LE(F^{s_1+1})$ for $s_2<s_1$. Therefore, we prove
$$LE(\overrightarrow{{F}_{n,k}^1})<LE(\overrightarrow{{F}_{n,k}^2})<\cdots<
LE(\overrightarrow{{F}_{n,k}^{q}})<LE(\overrightarrow{{F}_{n,k}^{q+1}}).$$

\end{proof}

Next, we consider the extremal Laplacian energy of $\overrightarrow{C_{k+1}}$-free digraphs when $k\geq3$.

\medskip\medskip
\noindent\textbf{Proof of Theorem~\ref{th:exLEC}:}

For $k\geq3$, assume that $G^\ast$ is a $\overrightarrow{C_{k+1}}$-free digraph having the maximum Laplacian energy. Similarly to the proof of Lemma~\ref{le:exMC3} and Theorem~\ref{th:exLEC3}, we consider two cases: (i) $e(G^\ast)=\textnormal{ex}(n,\overrightarrow{C_{k+1}})$; (ii) $e(G^\ast)<\textnormal{ex}(n,\overrightarrow{C_{k+1}})$.

\medskip\medskip
\noindent (i) $e(G^\ast)=\textnormal{ex}(n,\overrightarrow{C_{k+1}})$.

From Theorem~\ref{th:exC}, we know
$$\textnormal{ex}(n,\overrightarrow{C_{k+1}})=e(\overrightarrow{F_{n,k}})=\frac{1}{2}n^2+\frac{k-2}{2}n-\frac{r(k-r)}{2},$$
and $\textnormal{EX}(n,\overrightarrow{C_{k+1}})=\overrightarrow{\mathcal{F}_{n,k}}$, for $k\geq3$.
From Lemma~\ref{le:ordering}, we have
$$LE(\overrightarrow{{F}_{n,k}^1})<LE(\overrightarrow{{F}_{n,k}^2})<\cdots<
LE(\overrightarrow{{F}_{n,k}^{q}})<LE(\overrightarrow{{F}_{n,k}^{q+1}}).$$
So, when $0<r<k$, $\overrightarrow{{F}_{n,k}^{q+1}}$ is a $\overrightarrow{C_{k+1}}$-free digraph having the maximum Laplacian energy; when $r=0$, $\overrightarrow{\mathcal{F}_{n,k}}$ only have one type $\overrightarrow{{F}_{n,k}^0}$, and $\overrightarrow{{F}_{n,k}^0}$ is a $\overrightarrow{C_{k+1}}$-free digraph having the maximum Laplacian energy.

\medskip\medskip
\noindent (ii) $e(G^\ast)<\textnormal{ex}(n,\overrightarrow{C_{k+1}})$.

If $e(G^\ast)<\textnormal{ex}(n,\overrightarrow{C_{k+1}})$, we also only consider the case when $0<r<k$; the case when $r=0$ is similar.

First, we consider the First Zagreb Index of $\overrightarrow{C_{k+1}}$-free digraphs. Similarly to the proof of Lemma~\ref{le:exMC3}, we also can get a new digraph $F'$ by deleting some arcs in $\overrightarrow{{F}_{n,k}^{q+1}}$ such that $e(F')=e(G^\ast)$. In $\overrightarrow{{F}_{n,k}^{q+1}}$, we start deleting arcs from the vertex with the smallest outdegree. The deleting operations similarly to the table~\ref{ta:ch-1} in Lemma~\ref{le:exMC3}.

For $F'$, we know $F'$ is a $\overrightarrow{C_{k+1}}$-free digraph and $M_1(\overrightarrow{{F}_{n,k}^{q+1}})>M_1(F')$. We then increase or decrease the arcs in $F'$ to obtain the new digraph $G'$, where $e(F')=e(G')$. Similar to the proof of Claim~\ref{cl:ch-1} in Lemma~\ref{le:exMC3}, if the outdegree descending sequence of $F'$ and $G'$ satisfy the majorization, then $M_1(F')>M_1(G')$; if the outdegree descending sequence of $F'$ and $G'$ do not satisfy the majorization, then $G'$ has $\overrightarrow{C_{k+1}}$. So, we can not find a $\overrightarrow{C_{k+1}}$-free digraph $G$, such that $M_1(F')\leq M_1(G)$.

So if $e(G^\ast)<\textnormal{ex}(n,\overrightarrow{C_{k+1}})$, $\overrightarrow{{F}_{n,k}^{q+1}}$ is a $\overrightarrow{C_{k+1}}$-free digraph having the maximum First Zagreb Index.

Next, we consider the Laplacian energy of $\overrightarrow{C_{k+1}}$-free digraphs. Similarly to the proof of Theorem~\ref{th:exLEC3}, we need to consider the impact of arc transformations on $c_2$. We do the operations on the arcs in $F'$:
\begin{center}
\textbf{delete an arc with a tail of $v_{i_1}^{j_1}$ and add an arc with a tail of $v_{i_2}^{j_2}$}.
\end{center}
Let $F_1$ be the new digraph. Adding $(v_{i_2}^{j_2},v_{i'_2}^{j'_2})$, $c_2(F')$ may remain unchanged or increase by 2; deleting $(v_{i_1}^{j_1},v_{i'_1}^{j'_1})$, $c_2(F^{\frac{n+1}{2}})$ may remain unchanged or decrease by 2. So $c_2(F_1)=c_2(F')+2$, $c_2(F_1)=c_2(F')-2$ or $c_2(F_1)=c_2(F')$.

Similarly to the proof of Claim~\ref{cl:ch-1} of Lemma~\ref{le:exMC3}, we will always find a $\overrightarrow{C_{k+1}}$ in $F_1$ if $i_1\geq i_2$. If $i_1<i_2$, then $d_{F_1}^+(v_{i_1}^{j_1})=d_{F'}^+(v_{i_1}^{j_1})-1$, $d_{F_1}^+(v_{i_2}^{j_2})=d_{F'}^+(v_{i_2}^{j_2})+1$, and $d_{F'}^+(v_{i_1}^{j_1})-d_{F'}^+(v_{i_2}^{j_2})\geq3$. We have
\begin{align*}
LE(F_1)&=M_1(F_1)+c_2(F_1)\\
&=M_1(F')-\left(d_{F'}^+(v_{i_1}^{j_1})\right)^2-\left(d_{F'}^+(v_{i_2}^{j_2})\right)^2
+\left(d_{F'}^+(v_{i_1}^{j_1})-1\right)^2+\left(d_{F'}^+(v_{i_2}^{j_2})+1\right)^2+c_2(F_1)\\
&\leq M_1(F')-2\left(d_{F'}^+(v_{i_1}^{j_1})-d_{F'}^+(v_{i_2}^{j_2})\right)+2+c_2(F')+2\\
&=LE(F')-2\left(d_{F'}^+(v_{i_1}^{j_1})-d_{F'}^+(v_{i_2}^{j_2})\right)+4\\
&\leq LE(F')-6+4\\
&=LE(F')-2.
\end{align*}
So $LE(F_1)<LE(F')$.

Since
$$
LE(\overrightarrow{{F}_{n,k}^{q+1}})=M_1(\overrightarrow{{F}_{n,k}^{q+1}})+c_2(\overrightarrow{{F}_{n,k}^{q+1}})
>M_1(F')+c_2(F')=LE(F')>LE(F_1),$$
$\overrightarrow{{F}_{n,k}^{q+1}}$ is a $\overrightarrow{C_{k+1}}$-free digraph having the maximum Laplacian energy if $e(G^\ast)<\textnormal{ex}(n,\overrightarrow{C_{k+1}})$.

\medskip\medskip
In conclusion, we obtain:

\noindent (a) if $0<r<k$, then
\begin{align*}
\textnormal{ex}_{LE}(n,\overrightarrow{C_{k+1}})&=LE(\overrightarrow{{F}_{n,k}^{q+1}})\\
&=\sum_{i=1}^q k(n-1-(i-1)k)^2+r(n-1-qk)^2+qk(k-1)+r(r-1)\\
&=\frac{1}{3}n^3+\left(\frac{k}{2}-1\right)n^2+\frac{k^2}{6}n+\frac{2}{3}r^3-\frac{k}{2}r^2-\frac{k^2}{6}r,
\end{align*}
and $\textnormal{EX}_{LE}(n,\overrightarrow{C_{k+1}})=\overrightarrow{{F}_{n,k}^{q+1}}$;

\noindent (b) if $r=0$, then
\begin{align*}
\textnormal{ex}_{LE}(n,\overrightarrow{C_{k+1}})&=LE(\overrightarrow{{F}_{n,k}^0})\\
&=\sum_{i=1}^q k(n-1-(i-1)k)^2+qk(k-1)\\
&=\frac{1}{3}n^3+\left(\frac{k}{2}-1\right)n^2+\frac{k^2}{6}n,
\end{align*}
and $\textnormal{EX}_{LE}(n,\overrightarrow{C_{k+1}})=\overrightarrow{{F}_{n,k}^0}$.

This completes the proof of Theorem~\ref{th:exLEC}.

$\hfill\square$

\section{Concluding remarks}\label{sec4}

In Sections~\ref{sec2} and~\ref{sec3} of this paper, we fully characterized the extremal digraphs in $\overrightarrow{C_{k+1}}$-free digraphs having the maximum Laplacian energy.
From the proof of Lemma~\ref{le:exMC3}, Theorems~\ref{th:exLEC} and \ref{th:exLEC3}, the principle of the proof is:
\begin{itemize}
  \item [1.] According to the Tur\'{a}n number results for $\overrightarrow{C_{k+1}}$-free digraphs, we first propose a conjecture about the extremal digraph in $\overrightarrow{C_{k+1}}$-free digraphs having the maximum Laplacian energy, denoted as $G^\ast$.
  \item [2.] Suppose that $G^\ast$ in $\overrightarrow{C_{k+1}}$-free digraphs has the maximum First Zagreb Index. Starting from the original digraph $G^\ast$, we obtain a new digraph $G'$ by deleting arcs and adding arcs, ensuring that the number of arcs remains unchanged. Using the Karamata's inequality, if the outdegree descending sequence of $G^\ast$ and the new digraph $G'$ satisfy the majorization, then $M_1(G^\ast)\leq M_1(G')$, but $G'$ must have $\overrightarrow{C_{k+1}}$; if the outdegree descending sequence of $G^\ast$ and $G'$ do not satisfy the majorization, then $M_1(G^\ast)>M_1(G')$, conforming to the assumption.
  \item [3.] Finally, consider the influence of $c_2$ on the Laplacian energy. When $k\geq3$, there is no effect. When $k=2$, we again find other extremal digraphs.
\end{itemize}

We found that the extremal digraphs in $\overrightarrow{C_{k+1}}$-free digraphs have the maximum Laplacian energy also have the maximum number of arcs. However, there are more extremal digraphs in $\overrightarrow{C_{k+1}}$-free digraphs with the maximum number of arcs. Could it be that similar results hold for other forbidden subdigraphs? This is also a problem we can explore further.

In~\cite{ZhLi}, Zhou and Li also determined the Tur\'{a}n number and the extremal digraphs for $\overrightarrow{P_{k+1}}$. So we leave an open  problem as follows. Interested readers may also consider other forbidden subdigraphs.

\noindent\begin{problem}\label{pr:ch-4.1} Characterize the extremal digraphs in $\overrightarrow{P_{k+1}}$-free digraphs having the maximum Laplacian energy.
\end{problem}

In addition, the usual spectral Tur\'{a}n problem considers the maximum adjacency spectral radius. What results can be obtained for the $\overrightarrow{C_{k+1}}$-free digraphs have the maximum adjacency spectral radius? Will its extremal digraphs still have the maximum Laplacian energy? That is another challenging open problem.

\noindent\begin{problem}\label{pr:ch-4.2} Characterize the extremal digraphs in $\overrightarrow{C_{k+1}}$-free digraphs having the maximum adjacency spectral radius, where $k\geq2$.
\end{problem}

Actually, In~\cite{BrLi,Dr}, the authors found the extremal digraphs of $\overrightarrow{C_2}$-free digraphs having the maximum adjacency spectral radius. Its extremal digraphs are differ from the extremal digraph having the maximum Laplacian energy, but have the maximum number of arcs. Therefore, in the structure of forbidding specific subdigraphs, what is the relationship among the extremal digraphs with the maximum number of arcs, the maximum Laplacian energy, and the maximum adjacency spectral radius? These are all problems worthy of our careful consideration.

\end{document}